\documentclass[11pt]{article}
\usepackage{amsmath,amssymb,amsthm,booktabs,longtable}
\usepackage[margin=1in]{geometry}
\usepackage{xcolor}
\usepackage[hidelinks]{hyperref}
\newtheorem{theorem}{Theorem}

\newcommand{\R}{\mathbb R}

\newcommand{\kap}{\kappa}
\newcommand{\abs}[1]{\left\lvert#1\right\rvert}
\title{Improved lattice-packing bounds in small dimensions}
\author{Nina Zubrilina}
\date{\today}
\begin{document}
\maketitle
\begin{abstract}
We strengthen the upper bound on the density of a lattice packing in small dimensions. A lattice cannot have too many short vectors; we modify the Cohn--Elkies linear program accordingly by allowing the auxiliary function to be positive on a short interval beyond the minimum distance. For $11 \leq n \leq 16$, our bound is tighter than the three-point bound of Cohn, de Laat and Salmon. For $n \in [17, 23] \cup [25, 40] \cup [43, 57] \cup [62, 75]\cup [83, 91]$, our bound lies below the published numerical values of the Cohn--Elkies bound.
\end{abstract}
\newcommand{\CertificateCount}{82}
\newcommand{\ImprovementCount}{67}
\newcommand{\ImprovementDimensions}{11--23, 25--40, 43--57, 62--75, 83--91}
\newcommand{\SmallestMargins}{$0.0117\%$ at $n=62$, $0.0167\%$ at $n=91$, $0.0177\%$ at $n=75$, $0.0859\%$ at $n=83$}
\section{Introduction}
Let $\Lambda\subset\R^n$ be a full-rank lattice normalized to have shortest nonzero vector of length $1$, and write
$\rho(\Lambda)=1/\operatorname{vol}(\R^n/\Lambda)$. Placing balls of radius $1/2$ at the lattice points produces a packing of density $v_n2^{-n}\rho(\Lambda)$, where
$v_n=\pi^{n/2}/\Gamma(n/2+1)$ is the unit ball volume. The optimal lattice packing density is
\[
 \Delta_n^{\rm lat}=\frac{v_n}{2^n}\sup_\Lambda\rho(\Lambda).
\]

The Cohn--Elkies bound~\cite{CE} controls density through an auxiliary
function with nonnegative Fourier transform. In the lattice case,
Poisson summation bounds $\rho(\Lambda)\widehat f(0)$ by
$\sum_{v\in \Lambda}f(v)$, and the condition $f(x)\le0$ for $|x|\ge1$
makes the latter sum at most $f(0)$. 

For a lattice packing, we obtain additional constraints on which short vectors can cooccur. Martinet~\cite{Martinet} observed that each class of $\Lambda/2 \Lambda$ contains at most one pair $\{\pm v\}$ of length less than $\sqrt2$; pairs with lengths slightly above $\sqrt2$ have nearly orthogonal
directions, which allows us to bound their count as well. These restrictions allow
us to bound the total positive contribution from a class. 

Our modification always does at least as well as the Cohn--Elkies program, and in fact one can show that for $n$ sufficiently large, the two programs are equivalent. The improvement we get is in small dimensions, with diminishing returns as $n$ grows.  Further refinements of the class-counting constraints, including relations among classes modulo higher powers of \(2\), give small additional improvements; we omit those in this paper as they do not substantially narrow the gap to the best known lattice densities.

Section~\ref{sec:class-bound} states and proves the resulting bound.
Section~\ref{sec:certificates} describes the polynomial certificates
and the method of their verification. Finally, Section~\ref{sec:results}
presents the certified bounds and numerical comparisons.

The numerical exploration of the linear program was performed by GPT-$6$ Astra.

\section{The modified linear program}\label{sec:class-bound}

Let $n\ge2$ and $\sqrt2\le s<\min\{2,\sqrt{2/(1-n^{-1/2})}\}$.
For a nonzero class $C$ of $\Lambda/2\Lambda$ and $t\ge1$, we want to bound
\[
 N_C(t):=\frac12\#\{v\in C:|v|<t\},
\]
the number of opposite pairs in $C$ of length less than $t$.
We will bound $N_C(s)$ by $B(s)$, and, when $C$ contains a pair shorter
than $a$, bound the number of remaining pairs shorter than $s$ by $A(a,s)$.
To that end, let $c(s)=1-2/s^2$, so that $0\le c(s)<1/\sqrt n$, and define
\begin{equation}\label{eq:B}
 B(s)=\left\lfloor\frac{n(1-c(s)^2)}{1-nc(s)^2}\right\rfloor.
\end{equation}
For $1\le a\le\sqrt2$, set $A(a,s)=0$
when $a^2+s^2\le4$. Otherwise put
\[
 c_a=\frac{a^2+s^2-4}{2as},\qquad
 c'_a=\frac{c(s)+c_a^2}{1-c_a^2},
\]
and define
\begin{equation}\label{eq:A}
 A(a,s)=\begin{cases}
 \displaystyle\min\left\{B(s)-1,
 \left\lfloor\frac{(n-1)(1-(c'_a)^2)}{1-(n-1)(c'_a)^2}\right\rfloor\right\},
 &(n-1)(c'_a)^2<1,\\
 B(s)-1,&(n-1)(c'_a)^2\ge1.
 \end{cases}
\end{equation}

We will use the following class of step functions $G$. Choose an integer $M$ and partition points
$1=e_0<\cdots<e_M<\min\{2,\sqrt{2/(1-n^{-1/2})}\}$, with $e_J=\sqrt2$ for some $0<J<M$. Given heights $g_0,\ldots,g_{M-1}\ge0$, define
the step function $G:[0,\infty)\to[0,\infty)$ by
\begin{equation}\label{eq:weight}
 G(t)=\begin{cases}
 g_b,&e_b\le t<e_{b+1},\quad 0\le b<M,\\
 0,&\text{otherwise}.
 \end{cases}
\end{equation}
We require $g_J\ge\cdots\ge g_{M-1}$ and set $g_M=0$; no monotonicity
is imposed on the heights below $\sqrt2$. 

We define the \emph{$G$-budget} $\kap(G)$ by
\begin{align}
 \kap_\emptyset&=\sum_{b=J}^{M-1}(g_b-g_{b+1})B(e_{b+1}),\nonumber\\
 \kap_j&=g_j+\sum_{b=J}^{M-1}(g_b-g_{b+1})A(e_{j+1},e_{b+1})
 \quad(0\le j<J),\label{eq:kappa}\\
 \kap(G)&=\max\{\kap_\emptyset,\kap_0,\ldots,\kap_{J-1}\}.\nonumber
\end{align}
As proved below, this is a uniform upper bound on the total $G$-weight
of opposite pairs in any nonzero class $C$ of $\Lambda/2\Lambda$:
\[
 \frac12\sum_{v\in C}G(|v|)\le\kap(G).
\]
The factor $1/2$ counts each opposite pair once.
We normalize the Fourier transform as
\[
 \widehat h(y)=\int_{\R^n}h(x)e^{-2\pi i\langle x,y\rangle}\,dx.
\]

\begin{theorem}\label{thm:main}
Let $\Lambda\subset\R^n$ be a full-rank lattice of minimum length $1$, and let $G$ be as above. Suppose $h$ is a real
Schwartz function such that $\widehat h\ge0$, $\widehat h(0)>0$, and
\[
 h(x)\le G(\abs{x})\quad(1\le|x|<e_M),\qquad
 h(x)\le0\quad(|x|\ge e_M).
\]
Then
\begin{equation}\label{eq:main}
 \rho(\Lambda)\le\frac{h(0)+2(2^n-1)\kap(G)}{\widehat h(0)}.
\end{equation}
\end{theorem}

\begin{proof}
Following Cohn--Elkies, Poisson summation gives 
\[
 \rho(\Lambda)\widehat h(0)
 \le\rho(\Lambda)\sum_{w\in \Lambda^*}\widehat h(w)
 =\sum_{v\in \Lambda}h(v)
\]
and our goal is to show 
\[\sum_{v\in \Lambda \setminus \{0\}}h(v) \leq   2(2^n-1)\kap(G).\]

Suppose nonzero vectors $v,w\in\Lambda$ belong to the same class in
$\Lambda/2\Lambda$, with $v\ne\pm w$. Then $(v+w)/2$ and $(v-w)/2$
are in $\Lambda\setminus\{0\}$, hence
\begin{equation}\label{eq:coset-inner-product}
 |v\pm w|\ge2,\qquad
 |\langle v,w\rangle|\le\frac{|v|^2+|w|^2-4}{2}.
\end{equation}
In particular, $v$ and $w$ cannot both have length below $\sqrt2$.

The function $F(x,y)=(x^2+y^2-4)/(2xy)$ increases in each variable for
$1\le x,y<2$. Thus, for $|v|,|w|<s$,  $v \neq \pm w$, 
\[
 \left|\left\langle{v}/{|v|},{w}/{|w|}\right\rangle\right|
 \le F(|v|,|w|)\le F(s,s)=c(s).
\]
Now, let $u_1,\ldots,u_N\in\R^d$ be unit column vectors
with $|\langle u_i,u_j\rangle|\le c$ for $i\ne j$ and $dc^2<1$.
The positive semidefinite matrix $S=\sum_i u_i u_i^T$ satisfies
$\operatorname{tr}S=N$ and
$\operatorname{tr}(S^2)=\sum_{i,j}\langle u_i,u_j\rangle^2$.
Cauchy--Schwarz applied to its $d$ eigenvalues gives
\[
 N^2=(\operatorname{tr}S)^2\le d\operatorname{tr}(S^2)
 \le d\bigl(N+N(N-1)c^2\bigr)
\]
so $N\le d(1-c^2)/(1-dc^2)$. Hence
\begin{equation}\label{eq:class-count-B}
 N_C(s)\le B(s).
\end{equation}

Now suppose $N_C(a)=1$ for some $1\le a\le\sqrt2$, and let
$\{\pm v\}$ be this short pair. If $a^2+s^2\le4$, inequality~\eqref{eq:coset-inner-product}
excludes any additional pair of length less than $s$. Otherwise, let $u=v/|v|$, and choose one vector
$v_i$ from each of the $N_C(s)-1$ other opposite pairs of length less
than $s$ in $C$.
Monotonicity of $F$ gives $|\langle u_i,u\rangle|\le F(a,s)=c_a<1$
for $u_i=v_i/|v_i|$; here $c_a<1$ follows from $(s-a)^2<4$. For $\eta_i=\langle u_i,u\rangle$, the
normalized orthogonal projection of $u_i$ onto $u^\perp$ is
$z_i=(u_i-\eta_i u)/\sqrt{1-\eta_i^2}$. For $i\ne\ell$,
\[
 |\langle z_i,z_\ell\rangle|
 =\frac{|\langle u_i,u_\ell\rangle-\eta_i\eta_\ell|}
 {\sqrt{(1-\eta_i^2)(1-\eta_\ell^2)}}
 \le\frac{c(s)+c_a^2}{1-c_a^2}=c'_a.
\]
The projections are well defined because $|\eta_i|<1$. Applying the
relative bound for $d = n-1$ for $u^\perp$ when
$(n-1)(c'_a)^2<1$, and combining it with
$N_C(s)-1\le B(s)-1$ from~\eqref{eq:class-count-B}, gives
\begin{equation}\label{eq:class-count-A}
 N_C(s)-1\le A(a,s)\qquad\text{if }N_C(a)=1.
\end{equation}
When $(n-1)(c'_a)^2\ge1$, only the bound $B(s)-1$ is used;
when $a^2+s^2\le4$, the left side is zero.

Finally, we bound the sum of the auxiliary function over each nonzero
class $C$ of $\Lambda/2\Lambda$. Let $R_C$ contain one representative
of each opposite pair in $C$, and call a pair short if its length is
less than $\sqrt2$. Let $m_\ell$ count the pairs with lengths in
$[e_\ell,e_{\ell+1})$ for $J\le\ell<M$, and put
$N_b=\sum_{\ell=J}^b m_\ell=N_C(e_{b+1})-N_C(\sqrt2)$.
Summation by parts gives
\[
 \sum_{\ell=J}^{M-1}g_\ell m_\ell
 =\sum_{b=J}^{M-1}(g_b-g_{b+1})N_b.
\]
Since $g_b-g_{b+1}\ge0$, we may substitute upper bounds for $N_b$.
If $C$ has no short pair, then $N_C(\sqrt2)=0$, and
\eqref{eq:class-count-B} gives $N_b=N_C(e_{b+1})\le B(e_{b+1})$, so
\[
 \sum_{v\in R_C}G(|v|)
 =\sum_{\ell=J}^{M-1}g_\ell m_\ell
 \le\sum_{b=J}^{M-1}(g_b-g_{b+1})B(e_{b+1})
 =\kap_\emptyset.
\]
If its unique short pair lies in $[e_j,e_{j+1})$, then
$N_C(e_{j+1})=N_C(\sqrt2)=1$, and \eqref{eq:class-count-A} gives
$N_b=N_C(e_{b+1})-1\le A(e_{j+1},e_{b+1})$. Including that pair gives
\[
 \sum_{v\in R_C}G(|v|)
 =g_j+\sum_{\ell=J}^{M-1}g_\ell m_\ell
 \le g_j+\sum_{b=J}^{M-1}(g_b-g_{b+1})A(e_{j+1},e_{b+1})
 =\kap_j.
\]
Thus 
\[
 \sum_{v\in C}h(v) \le2\sum_{v\in R_C}G(|v|)\le2\kap(G).
\]
Every nonzero vector in $2\Lambda$ has length at least $2>e_M$ and hence contributes nonpositively, so summing over the $2^n-1$ nonzero classes and applying the opening Poisson inequality gives $\rho(\Lambda)\widehat h(0)\le h(0)+2(2^n-1)\kap(G)$, which proves the theorem.
\end{proof}

\section{Exact polynomial certificates}\label{sec:certificates}

For each dimension, we provide a \emph{certificate}: rational data that a
verification program, or \emph{verifier}, checks using exact arithmetic
to establish an upper bound on lattice packing density in $\R^n$ via
Theorem~\ref{thm:main}.

We construct radial auxiliary functions using the generalized Laguerre
polynomials $L_k^{(\alpha)}$. Set $\alpha=n/2-1$ and write
\[
 p(t)=\sum_{k=0}^K c_kL_k^{(\alpha)}(t),\qquad
 q(t)=\sum_{k=0}^K(-1)^kc_kL_k^{(\alpha)}(t).
\]
For $f(x)=p(2\pi|x|^2)e^{-\pi|x|^2}$, the Fourier transform is
$\widehat f(y)=q(2\pi|y|^2)e^{-\pi|y|^2}$, by the Fourier
eigenfunction identity in~\cite[Section~7]{CE}.

A certificate for a density bound consists of rational numbers $c_k$, $\tau_0>0$,
$E_b=e_b^2$, and $\Gamma_b\ge0$. Set
$\tau_b=\tau_0E_b$ and define a rescaling
\[
 h(x)=f\left(\sqrt{\frac{\tau_0}{2\pi}}\,x\right)
     =p(\tau_0|x|^2)e^{-\tau_0|x|^2/2}.
\]
Fourier scaling gives
\[
 \widehat h(y)=\left(\frac{2\pi}{\tau_0}\right)^{n/2}
 \widehat f\left(\sqrt{\frac{2\pi}{\tau_0}}\,y\right).
\]
The annulus $e_b\le|x|\le e_{b+1}$ now corresponds to
$\tau_b\le\tau_0|x|^2\le\tau_{b+1}$.

The verifiers check the partition-point conditions from Section~\ref{sec:class-bound} together with the following polynomial inequalities
\begin{align}
 q(t)&>0&& (t\ge0),\nonumber\\
 \Gamma_b\left(1+\frac{t-\tau_b}{2}+\frac{(t-\tau_b)^2}{8}\right)-p(t)&\ge0
 &&(\tau_b\le t\le\tau_{b+1}),\label{eq:checks}\\
 p(t)&<0&&(t\ge\tau_M).\nonumber
\end{align}
The middle condition is imposed for every $0\le b<M$, including both
endpoints. With $u=(t-\tau_b)/2\ge0$, the inequality
$1+u+u^2/2\le e^u$ shows that it implies
\[
 p(t)e^{-t/2}\le\Gamma_b e^{-\tau_b/2}.
\]
Thus these polynomial inequalities imply the required pointwise bounds
on $h$, as well as strict positivity of $\widehat h$.

To evaluate the budget $\kap(G)$ in~\eqref{eq:kappa} using rational arithmetic, put
\[
 R_b=\Gamma_b\left(\sum_{j=0}^{512}
                  \frac{(\tau_b/2)^j}{j!}\right)^{-1},\qquad
 g_b=\begin{cases}
 R_b,&b<J,\\
 \max_{b\le\ell<M}R_\ell,&b\ge J.
 \end{cases}
\]
The truncated exponential series gives
$R_b\ge\Gamma_be^{-\tau_b/2}$, and the backward maxima ensure the
required monotonicity without decreasing any height. Moreover,
\eqref{eq:B}--\eqref{eq:A} can be evaluated exactly from the squared
partition points: for example,
$c_a^2=(a^2+s^2-4)^2/(4a^2s^2)$ is rational. Consequently the resulting
$\kap$ and
\[
 \rho_+=\frac{p(0)+2(2^n-1)\kap}{q(0)}
\]
are rational. Since $\widehat h(0)=(2\pi/\tau_0)^{n/2}q(0)$,
Theorem~\ref{thm:main} gives
$\rho(\Lambda)\le(\tau_0/(2\pi))^{n/2}\rho_+$. Thus
\begin{equation}\label{eq:claims}
 \Delta_n^{\rm lat}\le
 \frac{\tau_0^{n/2}\rho_+}{8^{n/2}\Gamma(n/2+1)}.
\end{equation}

The reported decimal bounds are tested exactly. For even $n$, the
density expression in \eqref{eq:claims} is rational. For odd $n=2m+1$,
it is
\[
 \frac{2\tau_0^m\rho_+}{4^m(2m+1)!!}
       \sqrt{\frac{\tau_0}{8\pi}}.
\]
Choose a rational number $0<\xi<\pi$. A positive proposed
density bound in odd dimension is checked by squaring after replacing
$\pi$ by $\xi$. Positivity of each proposed decimal is checked before
squaring. The lower bound $\xi$ is
obtained from $\pi=16\arctan(1/5)-4\arctan(1/239)$ using the appropriate
upper and lower alternating Taylor sums.

Floating-point linear programming supplies candidate coefficients;
the exact sign tests and decimal comparisons establish their validity.
Following the data-and-verifier presentation of Cohn--Kumar~\cite{CK},
the exact certificates and verification code are available in the
\href{https://github.com/AstroNinja389/lattice-packing-bound}{\textcolor{blue}{\underline{GitHub repository}}}.
Four independent implementations use real-root counting, integer Taylor
bounds, Bernstein coefficients, and Descartes' rule of signs; the last
applies to even dimensions. Each reconstructs the polynomials, counting
bounds, and budget from the certificate data.

\clearpage
\section{Computational results}\label{sec:results}
Table~\ref{tab:results} lists the certified bounds in the
\ImprovementCount{} dimensions with a claimed improvement:
\ImprovementDimensions, together with comparisons in dimensions $9$ and $10$.
For $9\le n\le16$, the comparison values are the rigorous three-point bounds in
Cohn--de Laat--Salmon~\cite[Table~1.1]{CDS}. Tutunov et al.~\cite[Table~1]{Tutunov}
report slightly smaller numerical values; we use the rigorously verified
values of~\cite{CDS} for comparison.
For $n\ge17$, we use numerical estimates from the highest available
truncations of the dual solutions in~\cite{AJ} (in dimensions $17$--$91$,
the estimates from the two largest truncations differ by less than
$10^{-20}$).

The lower-bound column records the best lattice densities reported in the
Nebe--Sloane catalogue~\cite{NSCatalogue}, with the improvements in dimensions
$38$--$40$ from~\cite{DSW}, and the entries in dimensions $54$, $56$, $64$,
and $72$ from~\cite{Leterrier,CohnTable}.
We convert center density $\delta$ to packing density $v_n\delta$ and round
down to six significant digits. Nonlattice records are excluded; a dash
means that no sufficiently substantiated lattice-record value is supplied.

For $n = 9$, Dutour Sikiri\'c and
van Woerden's preprint~\cite[Corollary~1]{DutourVanWoerden} establishes
$\Delta_9^{\rm lat}=\sqrt2\,\pi^4/945$. Together with Mordell's inequality,
this yields
$\Delta_{10}^{\rm lat}\le\pi^5/(120\cdot2^{35/8})<0.122903$, so our bounds do not improve the strongest available lattice bounds in dimensions $9$ and $10$.

\begingroup
\small
\setlength{\tabcolsep}{4pt}
\begin{longtable}{crrrr}
\caption{Lattice packing density bounds. The comparison bounds are rigorous
three-point bounds for $9\le n\le16$ and numerical Cohn--Elkies values for
$n\ge17$. Gains are relative to this comparison column.}\label{tab:results}\\
\toprule $n$ & Best known lattice & Our bound & \shortstack{Three-point bound/\\Cohn--Elkies bound} & Gain (\%)\\\midrule
\endfirsthead
\toprule $n$ & Best known lattice & Our bound & \shortstack{Three-point bound/\\Cohn--Elkies bound} & Gain (\%)\\\midrule
\endhead
\bottomrule\endfoot
9 & $0.145774$ & $0.184380634$ & $0.191121$ & $3.527$\\
10 & $0.0920211$ & $0.132492965$ & $0.143411$ & $7.613$\\
11 & $0.0604326$ & $0.0949892088$ & $0.106726$ & $11.00$\\
12 & $0.0494541$ & $0.0680145536$ & $0.079712$ & $14.67$\\
13 & $0.0292084$ & $0.0488231948$ & $0.060165$ & $18.85$\\
14 & $0.0216240$ & $0.0352392080$ & $0.045062$ & $21.80$\\
15 & $0.0168575$ & $0.0256118267$ & $0.033757$ & $24.13$\\
16 & $0.0147081$ & $0.0187461206$ & $0.024995$ & $25.00$\\
17 & $0.00881131$ & $0.0138839313$ & $0.01846409034$ & $24.81$\\
18 & $0.00592836$ & $0.0103704834$ & $0.01348534045$ & $23.10$\\
19 & $0.00412080$ & $0.00782447361$ & $0.009817955140$ & $20.30$\\
20 & $0.00322586$ & $0.00595808443$ & $0.007127053603$ & $16.40$\\
21 & $0.00246588$ & $0.00456045297$ & $0.005159660395$ & $11.61$\\
22 & $0.00212766$ & $0.00348596988$ & $0.003725941969$ & $6.441$\\
23 & $0.00190532$ & $0.00264134013$ & $0.002684279886$ & $1.600$\\
25 & $0.000677212$ & $0.00132930810$ & $0.001384190722$ & $3.965$\\
26 & $0.000269220$ & $0.000915119176$ & $0.0009910238892$ & $7.659$\\
27 & $0.000128675$ & $0.000631342154$ & $0.0007082297959$ & $10.86$\\
28 & $6.97587\times10^{-5}$ & $0.000437000598$ & $0.0005052542161$ & $13.51$\\
29 & $2.78789\times10^{-5}$ & $0.000303842544$ & $0.0003598581852$ & $15.57$\\
30 & $1.44286\times10^{-5}$ & $0.000212402597$ & $0.0002559028744$ & $17.00$\\
31 & $1.18377\times10^{-5}$ & $0.000149408588$ & $0.0001817083814$ & $17.78$\\
32 & $1.10407\times10^{-5}$ & $0.000105764483$ & $0.0001288432888$ & $17.91$\\
33 & $4.14068\times10^{-6}$ & $7.55167890\times10^{-5}$ & $9.123560390\times10^{-5}$ & $17.23$\\
34 & $1.76697\times10^{-6}$ & $5.43081491\times10^{-5}$ & $6.452219674\times10^{-5}$ & $15.83$\\
35 & $9.46190\times10^{-7}$ & $3.93266737\times10^{-5}$ & $4.557438431\times10^{-5}$ & $13.71$\\
36 & $6.16146\times10^{-7}$ & $2.86680609\times10^{-5}$ & $3.215305533\times10^{-5}$ & $10.84$\\
37 & $3.21356\times10^{-7}$ & $2.09981537\times10^{-5}$ & $2.265869001\times10^{-5}$ & $7.328$\\
38 & $2.17566\times10^{-7}$ & $1.52662254\times10^{-5}$ & $1.595064991\times10^{-5}$ & $4.291$\\
39 & $1.68028\times10^{-7}$ & $1.09486830\times10^{-5}$ & $1.121686870\times10^{-5}$ & $2.391$\\
40 & $1.47983\times10^{-7}$ & $7.82816106\times10^{-6}$ & $7.880105170\times10^{-6}$ & $0.6592$\\
43 & $4.01571\times10^{-8}$ & $2.70203239\times10^{-6}$ & $2.716907473\times10^{-6}$ & $0.5475$\\
44 & $3.20085\times10^{-8}$ & $1.86230595\times10^{-6}$ & $1.901770314\times10^{-6}$ & $2.075$\\
45 & $2.52315\times10^{-8}$ & $1.28034519\times10^{-6}$ & $1.330090566\times10^{-6}$ & $3.740$\\
46 & $2.27179\times10^{-8}$ & $8.80676408\times10^{-7}$ & $9.295151556\times10^{-7}$ & $5.254$\\
47 & $2.14660\times10^{-8}$ & $6.07126117\times10^{-7}$ & $6.490757337\times10^{-7}$ & $6.463$\\
48 & $2.31782\times10^{-8}$ & $4.19948220\times10^{-7}$ & $4.529067790\times10^{-7}$ & $7.277$\\
49 & \textemdash & $2.91698991\times10^{-7}$ & $3.157963968\times10^{-7}$ & $7.631$\\
50 & \textemdash & $2.03458140\times10^{-7}$ & $2.200392958\times10^{-7}$ & $7.536$\\
51 & \textemdash & $1.42487373\times10^{-7}$ & $1.532136244\times10^{-7}$ & $7.001$\\
52 & \textemdash & $1.00209490\times10^{-7}$ & $1.066123391\times10^{-7}$ & $6.006$\\
53 & \textemdash & $7.07065699\times10^{-8}$ & $7.413767463\times10^{-8}$ & $4.628$\\
54 & $1.46699\times10^{-10}$ & $4.97631903\times10^{-8}$ & $5.152285374\times10^{-8}$ & $3.415$\\
55 & \textemdash & $3.48109436\times10^{-8}$ & $3.578475170\times10^{-8}$ & $2.721$\\
56 & $5.35530\times10^{-11}$ & $2.44093782\times10^{-8}$ & $2.483937099\times10^{-8}$ & $1.731$\\
57 & \textemdash & $1.71634335\times10^{-8}$ & $1.723196004\times10^{-8}$ & $0.3977$\\
62 & \textemdash & $2.74632309\times10^{-9}$ & $2.746644716\times10^{-9}$ & $0.01171$\\
63 & \textemdash & $1.89257793\times10^{-9}$ & $1.899489162\times10^{-9}$ & $0.3638$\\
64 & $1.32606\times10^{-12}$ & $1.29978772\times10^{-9}$ & $1.312998022\times10^{-9}$ & $1.006$\\
65 & \textemdash & $8.92325642\times10^{-10}$ & $9.071711718\times10^{-10}$ & $1.636$\\
66 & \textemdash & $6.13061198\times10^{-10}$ & $6.264943126\times10^{-10}$ & $2.144$\\
67 & \textemdash & $4.21802530\times10^{-10}$ & $4.324664005\times10^{-10}$ & $2.466$\\
68 & \textemdash & $2.90120314\times10^{-10}$ & $2.984004045\times10^{-10}$ & $2.775$\\
69 & \textemdash & $2.00732130\times10^{-10}$ & $2.058081299\times10^{-10}$ & $2.466$\\
70 & \textemdash & $1.38650284\times10^{-10}$ & $1.418880624\times10^{-10}$ & $2.282$\\
71 & \textemdash & $9.60867646\times10^{-11}$ & $9.778075467\times10^{-11}$ & $1.732$\\
72 & $1.45861\times10^{-13}$ & $6.66598519\times10^{-11}$ & $6.735796328\times10^{-11}$ & $1.036$\\
73 & \textemdash & $4.62232194\times10^{-11}$ & $4.638272400\times10^{-11}$ & $0.3439$\\
74 & \textemdash & $3.18709121\times10^{-11}$ & $3.192705609\times10^{-11}$ & $0.1759$\\
75 & \textemdash & $2.19646063\times10^{-11}$ & $2.196849358\times10^{-11}$ & $0.01769$\\
83 & \textemdash & $1.08926359\times10^{-12}$ & $1.090200461\times10^{-12}$ & $0.08594$\\
84 & \textemdash & $7.47188544\times10^{-13}$ & $7.478849669\times10^{-13}$ & $0.09312$\\
85 & \textemdash & $5.10524957\times10^{-13}$ & $5.128979896\times10^{-13}$ & $0.4627$\\
86 & \textemdash & $3.48807350\times10^{-13}$ & $3.516392699\times10^{-13}$ & $0.8053$\\
87 & \textemdash & $2.38928356\times10^{-13}$ & $2.410106753\times10^{-13}$ & $0.8640$\\
88 & \textemdash & $1.63824494\times10^{-13}$ & $1.651391388\times10^{-13}$ & $0.7961$\\
89 & \textemdash & $1.12348125\times10^{-13}$ & $1.131203657\times10^{-13}$ & $0.6827$\\
90 & \textemdash & $7.71724614\times10^{-14}$ & $7.746593987\times10^{-14}$ & $0.3788$\\
91 & \textemdash & $5.30260966\times10^{-14}$ & $5.303494178\times10^{-14}$ & $0.01668$\\
\end{longtable}
\endgroup


\clearpage
\begin{thebibliography}{99}
\raggedright
\bibitem{AJ}
N.~Afkhami-Jeddi, H.~Cohn, T.~Hartman, D.~de Laat, and A.~Tajdini,
\emph{High-dimensional sphere packing and the modular bootstrap},
J. High Energy Phys. (2020), no.~12, 066;
\href{https://arxiv.org/abs/2006.02560}{arXiv:2006.02560}.
Data: \url{https://hdl.handle.net/1721.1/125646}.
\bibitem{CE}
H.~Cohn and N.~Elkies, \emph{New upper bounds on sphere packings I},
Ann. of Math. \textbf{157} (2003), 689--714;
\href{https://arxiv.org/abs/math/0110009}{arXiv:math/0110009}.
\bibitem{CDS}
H.~Cohn, D.~de Laat, and A.~Salmon, \emph{Three-point bounds for sphere packing},
\href{https://arxiv.org/abs/2206.15373v1}{arXiv:2206.15373v1} (2022).
\bibitem{CK}
H.~Cohn and A.~Kumar, \emph{Optimality and uniqueness of the Leech lattice
among lattices}, Ann. of Math. \textbf{170} (2009), 1003--1050;
\href{https://arxiv.org/abs/math/0403263}{arXiv:math/0403263}.
\bibitem{Martinet}
J.~Martinet, \emph{Reduction modulo 2 and 3 of Euclidean lattices},
J. Algebra \textbf{251} (2002), 864--887.
Author's version: \url{https://jamartin.perso.math.cnrs.fr/Publications/Latmod.pdf}.
\bibitem{MartinetErr}
J.~Martinet, \emph{Lattices modulo 2 and 3 (comments)}, updated March~5, 2022,
\url{https://jamartin.perso.math.cnrs.fr/Publications/latmodcomm.pdf}.
\bibitem{Tutunov}
R.~Tutunov, A.~Maraval, A.~Grosnit, X.~Li, J.~Wang, and H.~Bou-Ammar,
\emph{Model-Based and Sample-Efficient AI-Assisted Math Discovery in Sphere Packing},
\href{https://arxiv.org/abs/2512.04829v2}{arXiv:2512.04829v2} (2025).
\bibitem{NSCatalogue}
G.~Nebe and N.~J.~A.~Sloane, \emph{Table of densest packings presently known},
Catalogue of Lattices,
\url{https://www.math.rwth-aachen.de/~Gabriele.Nebe/LATTICES/density.html}.
\bibitem{DSW}
I.~Dorofeev, X.~Sun, and C.~Wang,
\emph{Optimal extensions of cross-sections: Sphere packings in dimensions 38 to 43},
\href{https://arxiv.org/abs/2607.20359v4}{arXiv:2607.20359v4} (2026).
\bibitem{Leterrier}
G.~Leterrier, \emph{On the Mordell--Weil lattice of
$y^2=x^3+bx+t^{3^n+1}$ in characteristic 3},
Res. Number Theory \textbf{8} (2022), article 23,
\href{https://doi.org/10.1007/s40993-022-00321-0}{doi:10.1007/s40993-022-00321-0}.
\bibitem{CohnTable}
H.~Cohn, \emph{Table of sphere packing density bounds},
\url{https://cohn.mit.edu/sphere-packing/}.
\bibitem{DutourVanWoerden}
M.~Dutour Sikiri\'c and W.~van Woerden,
\emph{The lattice packing problem in dimension 9 by Voronoi's algorithm},
preprint (2025),
\href{https://arxiv.org/abs/2508.20719}{arXiv:2508.20719}.
\end{thebibliography}
\end{document}